\documentclass[leqno,11pt,oneside]{amsart}
\usepackage[dvips]{color}
\usepackage[hypertex]{hyperref}
\usepackage{graphicx}
\usepackage[abs]{overpic}

\usepackage{amsfonts}
\usepackage{amsthm}
\usepackage{amssymb}
\usepackage{bm}
\usepackage{enumerate}
\newcommand{\sing}{{\rm Sing}}

\usepackage[english]{babel}

\newtheorem{teo}{Theorem}[section]
\newtheorem{prop}[teo]{Proposition}
\newtheorem{ddef}[teo]{Definition}
\newtheorem{example}[teo]{Example}
\newtheorem{cor}[teo]{Corollary}

\newtheorem{lem}[teo]{Lemma}
\newtheorem*{cor*}{Corollary}

\newtheorem{maintheorem}{Theorem}

\newtheorem*{remark}{Remark}

\newtheorem*{lem*}{Lemma}

\newtheorem*{teorA'}{Theorem A'}

\theoremstyle{definition}

\newcommand{\R}{\mathbb{R}}

\newcommand{\Z}{\mathbb{Z}}

\newcommand{\C}{\mathbb{C}}

\newcommand{\A}{{\mathcal{A}}}

\newcommand{\F}{\mathcal{F}}
\newcommand{\LL}{{\mathcal L}}

\newcommand{\im}{{\rm Im}}
\newcommand{\re}{{\rm Re}}
\newcommand{\cl}[1]{\mbox{$\mathcal{#1}$}}
\newcommand{\ov}[1]{\mbox{$\overline{#1}$}}
\newcommand{\diml}{\dim_{\mathcal L}}
\newcommand{\bs}[1]{\mbox{$\boldsymbol{#1}$}}
\newcommand{\ind}{{\rm Ind}}

\newcommand*\xbar[1]{ %
   \hbox{ %
     \vbox{%
       \hrule height 0.3pt 
       \kern0.35ex
       \hbox{%
         \kern-0.1em
         \ensuremath{#1}%
         \kern-0.1em
       }%
     }%
   }%
}

\newcommand*\xxbar[1]{%
   \hbox{%
     \vbox{%
       \hrule height 0.3pt 
       \kern0.4ex
       \hbox{%
         \kern-0.1em
         \ensuremath{#1}%
         \kern-0.1em
       }%
     }%
   }%
}

 \newenvironment{prova}{\noindent {\bf Proof:}}
{\hfill \framebox[7pt]{} \mbox{} \medskip}

\begin{document}

\setcounter{section}{0}
\setcounter{teo}{0}
\setcounter{exe}{0}

\author{Jane Bretas \& Arturo Fern\'andez-P\'erez \& Rog\'erio   Mol}
\title{Holomorphic foliations tangent to Levi-flat subsets}
\maketitle

\begin{abstract}
We study Segre varieties associated to Levi-flat subsets in complex manifolds and apply them to establish local and global results on the integration of tangent holomorphic foliations.
\end{abstract}

\footnotetext[3]{First author partially financed by a CNPq Ph.D. fellowship. Second an third authors partially financed by CNPq-Universal}

 \medskip \medskip

\section{Introduction}
Let  $H \subset M$ be a real analytic hypersurface, where $M$ is a complex manifold of
$\dim_{\C} M = N$.  Let $H_{reg}$ denote its {\em regular} part, that is, the collection of all points near which $H$  is a manifold of maximal dimension.
For each   $p \in H_{reg}$, there is a unique complex hyperplane $\cl{L}_{p}$ contained in the tangent space $T_{p}H_{reg}$.
This  defines a real analytic distribution $p \mapsto \cl{L}_{p}$ of complex hyperplanes in $T H_{reg}$.  When this distribution is integrable in the sense of Frobenius, we say that $H$ is a {\em Levi-flat} hypersurface. The resulting foliation in $H$, denoted by $\cl{L}$,
is known as  {\em Levi foliation}. A normal form for  such an object was given by
E. Cartan \cite[Th. IV]{cartan1933}: for each $p \in H_{reg}$, there are   holomorphic coordinates $(z_{1},\ldots,z_{N})$ in a neighborhood $U$ of $p$
such that
\begin{equation}\label{formalocal-hlf}
H_{reg} \cap U = \{\im(z_{N}) = 0\}.
\end{equation}
 As a consequence, the leaves of
$\cl{L}$
 have local equations $z_{N} = c$, for $c \in \R$.

Cartan's local trivialization   allows the extension the Levi foliation  to a non-singular holomorphic foliation  in a neighborhood of  $H_{reg}$ in $M$, which is unique as a germ around $H_{reg}$.
In general, it is not possible to extend    $\cl{L}$   to a singular holomorphic foliation in a neighborhood of
$\overline{H_{reg}}$. There are examples of Levi-flat hypersurfaces whose Levi foliations extend to   $k$-webs in the ambient space \cite{brunella2007,fernandez2013}.  However
there is an extension  in some ``holomorphic lifting'' of $M$ \cite{brunella2007}.
If   a  foliation $\cl{F}$ in the ambient space $M$ coincides with the     Levi foliation on $H_{reg}$, we say either that $H$ is \emph{invariant} by $\cl{F}$ or that $\cl{F}$ is \emph{tangent} to $H$.

Locally, germs of codimension one foliations at $(\C^{N},0)$ tangent to real analytic   Levi-flat hypersurfaces  are   given by the levels of  meromorphic functions --- possibly holomorphic --- according to a  theorem by D. Cerveau and
A. Lins Neto  \cite{cerveau2011}.  Questions involving the global integrability of
codimension one foliations in $\mathbb{P}^{N}$ tangent to Levi-flat hypersurfaces where addressed
by J. Lebl in \cite{jiri2012}. For instance,  if $H$ is a real algebraic Levi-flat hypersurface
tangent to a codimension one foliation $\F$ in  $\mathbb{P}^{N}$, then $\F$ admits a rational first integral $R$ and there is a real algebraic curve $S \subset \C$ such that $H \subset \overline{R^{-1}(S)}$.

Our goal in this paper is to establish local and global integrability results for foliations tangent to
real analytic  \emph{Levi flat} subsets. A real analytic subset $H \subset M$, where $M$ is an $N$-dimensional complex manifold, is called
\emph{Levi-flat} if it has real dimension $2n + 1$ and its regular part $H_{reg}$ is foliated by
complex varieties of   dimension $n$ (Section \ref{levi-section}, Definition \ref{sclf}). This is called \emph{Levi foliation} and $n$  is  the \emph{Levi dimension} of $H$.  When $N = n+1$, we recover the definition of Levi-flat hypersurface.
This object appears in  M. Brunella's  study   on Levi-flat hypersurfaces \cite{brunella2007}, as the result of the lifting
of a Levi-flat hypersurface  to the
 the projectivized cotangent bundle of the ambient space by means of
the Levi foliation  (see Section \ref{section-examples}).

Key ingredients in the study of integrability properties of Levi-flat hypersurfaces are
\emph{Segre varieties}. Their structure  is    used  in Brunella's geometric  proof for the
local integrability of foliations tangent to Levi-flat hypersurfaces \cite{brunella2011} as well as  in Lebl's global integrability results  \cite{jiri2012}.   Segre varieties for   Levi-flat subsets are the cornerstone of our work. Their definition,
along with
  main properties,  are presented in Section \ref{segre-section}.
Recently, a research paper on Levi flat subsets, also founded on the study of Segre varieties,  has been released  \cite{pinchuk2016}. It has an  approach to Segre varieties   slightly different from ours, although leading to equivalent constructions.

Given a Levi-flat subset $H$ of Levi dimension $n$, there is a unique complex variety $H^{\imath}$ of dimension $n+1$,
called \emph{intrinsic complexification} or $\imath$-\emph{complexification},  defined in a neighborhood of $\overline{H_{reg}}$  containing $\overline{H_{reg}}$ \cite[Th. 2.5]{brunella2007}. If $H$ is tangent to a foliation $\F$ of dimension $n$ in the
ambient space, then $H^{\imath}$ is invariant by $\F$.
Our integration results are stated in terms of the $\imath$-complexification $H^{\imath}$ and  the foliation $\F^{\imath}$,
  the restriction of $\F$ to $H^{\imath}$.
For real analytic  Levi-flat subsets in projective spaces, we can state the following theorem, to be proved along Sections \ref{integration-section} and \ref{integration1-section}:

\begin{maintheorem}
\label{int-racional-R}
Let $H\subset \mathbb{P}^N$, $N>3$, be a real analytic Levi-flat subset  of Levi dimension $n$ invariant by a $n$-dimensional holomorphic foliation $\mathcal{F}$ on $\mathbb{P}^N$. Suppose that $n>(N-1)/2$. If the Levi foliation $\mathcal{L}$ has infinitely many algebraic leaves, then:
\begin{enumerate}
\item the $\imath$-complexification $H^{\imath}$ of $H$ extends to an algebraic variety in $\mathbb{P}^N$;
\item the   foliation $\mathcal{F}^{\iota}=\mathcal{F}|_{H^{\imath}}$ has a rational first integral $R$ in $H^{\imath}$;
\item there exists a real algebraic curve $S\subset \mathbb{C}$ such that $H\subset \overline{R^{-1}(S)}$. In particular $H$ is semialgebraic.
\end{enumerate}
\end{maintheorem}

For a real algebraic Levi-flat subset   $H\subset \mathbb{P}^N$, the $\imath$-complexification $H^{\imath}$ is algebraic.
If further $H$ is invariant by a global $n$ dimensional holomorphic foliation, then
the same elements  of the proof of Theorem \ref{int-racional-R}  give that assertions
\emph{(2)} and \emph{(3)} are also true in this case.

In the local point of view,  we   have the following     integrability result:

\begin{maintheorem}
\label{teo-bru-adapt}
 Let $\mathcal{F}$ be a germ of holomorphic foliation of dimension $n$  at $(\mathbb{C}^{N},0)$ tangent to a germ of real analytic Levi-flat subset  $H$ of Levi dimension $n$.  Suppose that $\sing(H^{\imath})$,   the singular set of the $\imath$-com\-plexi\-fi\-ca\-tion of $H$,  has codimension at least two.   Then $\mathcal{F}^{\imath}$ admits a meromorphic first integral.
\end{maintheorem}
The proof of this theorem, in Section \ref{section-brunella}, relies on   the    integration techniques of   Brunella's
 geometric proof   for Cerveau-Lins Neto's  local integrability theorem    \cite{brunella2011}.
Lastly,  we illustrate our main results with some examples in Section \ref{section-examples}.

This article is
a  partial compilation of the results of the Ph.D. thesis of the first author \cite{bretas2016},
written under the supervision of the second and third authors.
They all express their gratitude to
 R. Rosas and B. Sc\'ardua for    suggestions in the development of this work.

\section{Mirroring and complexification}
\label{preliminaries}

Consider   coordinates $z = (z_{1},\ldots,z_{N})$ in $\mathbb{C}^{N}$, where $z_{j} = x_{j} + i y_{j}$,
and the complex conjugation $\bar{z} = (\bar{z}_{1},\ldots,\bar{z}_{N})$, where $\bar{z}_{j} = x_{j} - i y_{j}$.
We will employ the standard multi-index  notation. For instance,   if $\mu = (\mu_{1},\ldots\mu_{N}) \in \Z_{\geq 0}^{N}$ then
$z^{\mu} = (z_{1}^{\mu_{1}},\ldots,z_{N}^{\mu_{N}})$.  We also fix the following notation
for rings of germs at   $(\C^{N},0)$:
\begin{itemize}
\item $\mathcal{O}_{N} =  \C\{z_{1},\ldots,z_{N}\}$ is the ring of germ of complex analytic functions;
\item $\A_{N} = \C\{z_{1},\ldots,z_{N},\bar{z}_{1},\ldots,\bar{z}_{N}\} = \C\{x_{1},y_{1},\ldots,x_{N},y_{N}\}$ is the ring of germs of real analytic functions with complex values;
\item $\A_{N \R} \subset \A_{N}$ is the ring of germs of real analytic functions with real values.
\end{itemize}
A  germ of function  $\phi(z)=\sum_{\mu,\nu} a_{\mu\nu} z^{\mu}\bar{z}^{\nu}$ in $\A_{N}$ is in $\A_{N \R}$
if and only if $\phi(z) = \xbar{\phi(z)}$ for all $z$, which is equivalent to $a_{\mu\nu} =
\bar{a}_{\nu \mu}$ for all $\mu, \nu$.

Let $\mathbb{C}^{N*}$ be the  space with the opposite complex structure of  $\mathbb{C}^{N}$, having complex coordinates $w = (w_{1},\ldots,w_{N}) = \bar{z}$.   The conjugation map $\Gamma: z=x+iy \mapsto x-iy = w$ defines a biholomorphism between $\mathbb{C}^{N}$ and $\mathbb{C}^{N*}$.
This correspondence is referred to as \textit{mirroring}. In general, given a subset $A \subset \C^{N}$,
its  \emph{mirror} is the set
$$A^{*} = \Gamma(A) = \{\bar{z} ;z\in A\}\subset \mathbb{C}^{N*}.$$
Given a complex function $\phi$ in $A \subset \mathbb{C}^{N}$, its mirror     $\phi^{*}$ is the  function
in $A^{*} \subset \mathbb{C}^{N*}$ given by
 $$\phi^{*}(w) =  \xxbar{\phi(\xbar{w})}.$$
 For instance, if $\phi(z)=\sum_{\mu}a_{\mu}z^{\mu}$ is
complex analytic, then its mirror
\[ \phi^{*}(w) = \displaystyle \xxbar{\sum_{\mu}a_{\mu}\bar{w}^{\mu}} = \sum_{\mu}\bar{a}_{\mu}w^{\mu} \]
is complex analytic. In the same way, if
 $\phi \in \A_{N \R}$  has a  development in power series
$\phi(z)=\sum_{\mu,\nu}a_{\mu\nu}z^{\mu}\bar{z}^{\nu},$ where $z \in  \mathbb{C}^N$, then
its mirror function $\phi^{*} \in \A_{N \R}$ has a power series  expansion
\begin{equation}
\label{mirror-function}
\phi^{*}(w)= \xxbar{ \sum_{\mu,\nu}a_{\mu\nu}\bar{w}^{\mu}w^{\nu}} =  \sum_{\mu,\nu}\bar{a}_{\mu\nu}w^{\mu}\bar{w}^{\nu} =
\sum_{\mu,\nu} a_{\nu\mu}w^{\mu}\bar{w}^{\nu},
\end{equation}
 where $w \in    \mathbb{C}^{N*}$.
It follows from this discussion that, if $A \subset \mathbb{C}^{N}$ is  a (real or complex) analytic subset, so is its mirror $A^{*} \subset \mathbb{C}^{N*}$.

This  mirroring procedure can  be applied to other geometric objects.
For example, to an    analytic $p-$form
$\eta=\sum_{I}\alpha_I(z)dz_I$,
where $I=(i_1,\ldots,i_p)$ and $dz_I=dz_{i_1}\wedge \cdots \wedge dz_{i_p}$, we  associate the  $p-$form $\eta^*=\sum_{I}\alpha_I^{*}(w)dw_I$.
A germ of holomorphic foliation $\mathcal{F}$ of codimension  $p$  at $(\mathbb{C}^{N},0)$,
defined by a $p-$form $\eta$ --- that is integrable and locally decomposable
  outside the singular set --- engenders its mirror $\mathcal{F}^{*}$, which is the foliation
   of codimension $p$ defined
 by  $\eta^*$ whose leaves are
   the mirroring of  those of $\mathcal{F}$ (see the Appendix for the definition of holomorphic foliation).

We consider $\mathbb{C}^{N} \times \mathbb{C}^{N*} \simeq \mathbb{C}^{2N} $ with
coordinates $(z,w)$,
  the embedding
\begin{eqnarray}\nonumber
i:\mathbb{C}^{N}&\rightarrow& \mathbb{C}^{N} \times \mathbb{C}^{N*} \nonumber \\
z&\mapsto&(z,\bar{z})\nonumber
\end{eqnarray}
and the diagonal subset
$$\Delta:=i(\mathbb{C}^{N})=\{(z,w)\in \mathbb{C}^{N}\times \mathbb{C}^{N*}; w=\bar{z}\}.$$
 Given a germ of analytic function $\phi \in \A_{N \R}$
  we say that a connected
neighborhood $U$ of $0 \in \mathbb{C}^{N}$ is  \emph{reflexive} for $\phi$ or $\phi$-\emph{reflexive}   if
 $\phi(z,w)$ converges in $U\times U^{*} \subset \mathbb{C}^{N}\times \mathbb{C}^{N*}$.
 For a germ of map $\bs{\phi}=(\phi_1,...,\phi_k) \in  (\A_{N \R})^{k}$, a $\bs{\phi}$-\emph{reflexive}
 neighborhood is one that is $\phi_{j}$-reflexive for every $j=1,...,k$.

Let $\phi \in \A_{N \R}$ be a real function with  development in power series
$\phi(z)=\sum_{\mu,\nu}a_{\mu\nu}z^{\mu}\bar{z}^{\nu}$. The \emph{complexification} of $\phi$ is the germ  of complex function
$ \phi^{ \C}  \in
\cl{O}_{2N}$ defined at the origin
$0  \in \mathbb{C}^{N}\times \mathbb{C}^{N*}$ by the series
\begin{equation}
\label{complexification-function}
\psi^{\C}(z,w)=\sum_{\mu,\nu}a_{\mu\nu}z^{\mu} w^{\nu}.
\end{equation}
  If $U \subset \mathbb{C}^{N}$ is a
$\phi$-reflexive neighborhood, then this series
converges in $U \times U^{*}$.
The complexification of   a germ of map $\bs{\phi}=(\phi_1,...,\phi_k) \in  (\A_{N \R})^{k}$
is the germ of complex map $\bs{\phi}^{\C} \in (\cl{O}_{2N})^{k}$ defined by $\bs{\phi}^{\C}=(\phi_1^{\C},...,\phi_k^{\C})$.

Let $H$ be a germ of real analytic variety at $(\mathbb{C}^{N},0)$. As before, we denote by
$H_{reg}$ its regular part. The singular part of $H$, denoted by $H_{sing}$,  consists of the points
in $H \setminus \overline{H_{reg}}$.
Let  $\mathcal{I}(H)$ denote the
ideal of $H$ in $\A_{N \R}$. Since  $\A_{N \R}$
is Noetherian, we can take a system of   generators $\phi_1,...,\phi_k$ of $\mathcal{I}(H)$ and associate a   map $\bs{\phi}=(\phi_1,...,\phi_k) \in  (\A_{N \R})^{k}$ that is called  \emph{generating map} of $H$.
We have the definition:
\begin{ddef}\label{complexificacao}
\emph{The \emph{extrinsic complexification} or simply \emph{complexification} $H^\mathbb{C}$ of $H$ is the germ of complex analytic variety at $(\mathbb{C}^{N}\times \mathbb{C}^{N*},0)$ defined by the equation
$ \bs{\phi}^{\C}(z,w)=0$.}
\end{ddef}
If $U$ is $\bs{\phi}$-\emph{reflexive}
neighborhood, then $H^\mathbb{C}$ is realized as
$$H^\mathbb{C} = \{(z,w)\in U \times U^{*}; \bs{\phi}(z,w)=0\}.$$
The set $H^{\mathbb{C}}$ is the smallest germ of complex analytic subset at $(\mathbb{C}^{N}\times \mathbb{C}^{N*},0)$ containing  $H_{\Delta} := i(H) = H^{\mathbb{C}}\cap \Delta$.
It is evident from the definition that the complexification respects inclusions:
if $H_{1} \subset H_{2}$ are   germs of real analytic varieties   then $H_{1}^{\C} \subset  H_{2}^{\C}$.
This notion of complexification, introduced by H. Cartan in \cite{cartan1957},   has the following properties:

\begin{enumerate}[(i)]
\item  $H^{\mathbb{C}}\supset H_{\Delta};$
\item  every germ of holomorphic function   vanishing over $H_{\Delta}$ also vanishes over $H^{\mathbb{C}};$
\item   the irreducible components of the real analytic variety $H$ are in correspondence, by complexification, to the irreducible components of the  complex analytic variety $H^{\mathbb{C}}$. In particular, $H$ is   irreducible if and only if $H^{\mathbb{C}}$ is irreducible.
\end{enumerate}

Let us examine the effect of the complexification procedure on  complex varieties.
Take $X \subset (\mathbb{C}^{N},0)$   a germ of complex analytic variety
whose ideal in $\cl{O}_{N}$ is generated by $f_{1}, \ldots, f_{k}$. Seen as a real analytic
variety, the corresponding generators of  the ideal of $X$ in $\A_{N \R}$ are
$\phi_{j} = \re(f_{j}) = (f_{j} + \bar{f}_{j})/2$ and $\psi_{j} = \im(f_{j}) = (f_{j} - \bar{f}_{j})/2 i$,
for $j=1,\ldots,k$. Thus, the complexification $X^{\C}$ in $(\mathbb{C}^{N}\times \mathbb{C}^{N*},0)$
is the complex analytic variety defined by the system of equations
\[\phi_{j}(z,w) =  \frac{f_{j}(z) + \bar{f}_{j}(\bar{w})}{2} = \frac{f_{j}(z) + {f}_{j}^{*}(w)}{2} = 0 \]
and
\[\psi_{j}(z,w) =  \frac{f_{j}(z) - \bar{f}_{j}(\bar{w})}{2i} = \frac{f_{j}(z) - {f}_{j}^{*}(w)}{2i} = 0,\]
 for $j=1,\ldots,k$, which is equivalent to
\[ f_{j}(z) = 0 \ \ \  \text{and} \ \ \ {f}_{j}^{*}(w) =0 \ \ \text{for} \ \ j=1,\ldots,k.\]
We therefore   conclude that $X^{\C} = X \times X^{*}$. In particular, we have that
$\dim_{\C} X^{\C}  = 2 \dim_{\C} X$.



\section{Levi-flat subsets, local aspects}
\label{levi-section}


Essentially,   real analytic Levi-flat subsets are real analytic subsets  of odd real dimension $2n+1$
  foliated by complex varieties of complex  dimension $n$. When the   real codimension is   one, we are in the case
of  Levi-flat hypersurfaces.
We give the precise definition:

\begin{ddef} {\rm \label{sclf}
Let $H\subset M$ be a  real analytic     subset of real dimension $2n+1$, where
 $M$ is an $N$-dimensional complex manifold, $N \geq 2$ and $1\leq n\leq N-1.$ We say that $H$ is a \emph{Levi-flat} subset if the distribution of tangent spaces
\begin{eqnarray}\nonumber
\mathcal{L}: H_{reg}\subset \mathbb{C}^N&\rightarrow& T\mathbb{C}^{N}\simeq\mathbb{C}^{N}\\ \nonumber
p&\mapsto&T_pH_{reg}\cap J(T_pH_{reg}) \nonumber
\end{eqnarray}
has dimension $n$ and is integrable in the sense of Frobenius.
}\end{ddef}
The  regular part of
  $H$ is a CR-variety, of CR-dimension $n+1$,
carrying  an $n-$dimensional
foliation with complex leaves.
We use the qualifier ``\emph{Levi}'' for the foliation, its leaves and its dimension.
The foliation itself is also   denoted by  $\cl{L}$, its dimension is called $\cl{L}$-dimension and denoted  by $\dim \cl{L}$ or $\diml H$.
 The leaf
through by $p \in H_{reg}$ is denoted  by $L_p$.
Also, we say that $N$ a the   \emph{ambient dimension} of  $H$.
Most of the time we are concerned with local properties of Levi-flat subsets. In this     case,
  an open set $U \subset \C^{N}$ plays the role of $M$ in the definition. The notion of
 Levi-flat subset germifies and,  in general, we do not distinguish a germ 
at $(\C^{N},0)$ from its realization in some neighborhood $U$ of $0 \in \C^{N}$.

A trivial   model for a   Levi-flat subset of $\cl{L}$-dimension $n$ in $\C^{N}$ is provided by
\begin{equation}
\label{levi-local-form}
H=\{z = (z',z'') \in \C^{n+1} \times \C^{N-n-1};\  \im(z_{n+1})=0,z''=0\},
\end{equation}
where $z'=(z_1, ... ,z_{n+1})$ and $z''=(z_{n+2},...,z_N)$. The Levi foliation is given by
$$\{z = (z',z'') \in \C^{n+1} \times \C^{N-n-1};\  z_{n+1}=c , z''=0,  \ \text{with} \ c \in \mathbb{R}\}.$$
This trivial model is in fact a local form for Levi-flat subsets.  This was mentioned in \cite{brunella2007} without an
explicit proof, which we give   for the sake of completeness:

\begin{prop}\label{est.loc.slf}
Let $H$ be a Levi-flat subset of $\cl{L}$-dimension $n$ and ambient dimension $N$.   Then, at each $p\in H_{reg}$, there are local holomorphic coordinates $(z',z'') \in \C^{n+1} \times \C^{N-n-1}$  such that $H$ has the local form
\eqref{levi-local-form}.
\end{prop}

\begin{proof}
Since $H_{reg}$ is a CR-subvariety, for some $k$ with $2 \leq k \leq N$, there are  local holomorphic coordinates
 $t=(t',t'')\in \mathbb{C}^{k} \times \mathbb{C}^{N-k}$ at $p$ such that $H_{reg}\subset \{t'' =0\}\cong \mathbb{C}^{k}$ is a generic subvariety, that is, $H_{reg}$ is defined by $d$ real functions in $\mathbb{C}^{k}$ whose  complex differentials are $\mathbb{C}$-linearly independent \cite[Cor. 1.8.10]{baouendi1999}.
  This gives
\[ dim_{\mathbb{R}}H_{reg}+d=2k \quad \text{and} \quad dim_\mathbb{C}T^{(1,0)}H_{reg}+d=k.\]
Combining these equations, we obtain
 $$k=dim_\mathbb{R}H_{reg}-dim_\mathbb{C}T^{(1,0)}H_{reg}= (2n +1) - n =n+1.$$
We found that $H_{reg}$ is as a real analytic Levi-flat hypersurface in the complex variety $\{t''=0\}$. It then suffices  to apply E. Cartan's normal form \eqref{formalocal-hlf}  to the coordinates $t'$
in order to get the  coordinates $z'$ and take $z'' = t''$.
\end{proof}

In  the local form \eqref{levi-local-form}, $\{z'' = 0\}$ corresponds to the unique local  $(n+1)-$dimensional complex subvariety of the ambient space containing the germ of $H_{reg}$ at $p$.  These local subvarieties glue together forming a complex variety defined in a whole neighborhood of $H_{reg}$.  It is analytically extendable  to a neighborhood of
$\ov{H_{reg}}$ by the following theorem:

\begin{teo}[Brunella \cite{brunella2007}]
\label{Levi-viz-Hi}
Let $M$ be an $N-$dimensional complex manifold and $H \subset M$ be a  real analytic Levi-flat subset of $\LL$-dimension $n$.  Then,   there exists  a neighborhood $ V \subset M$ of $\overline{H_{reg}}$   and  a unique complex    variety $X \subset V$ of dimension $n+1$ containing $H$.
\end{teo}
The variety $X$ is the realization in the neighborhood $V$   of a germ of complex analytic
variety around $H$.
We  denote it --- or its germ --- by $H^{\imath}$ and call it  \emph{intrinsic complexification} or  \emph{$\imath$-complexification} of $H.$ It plays a central role in the theory of Levi-flat subsets
we develop. The notion of intrinsic complexification also appears in \cite{pinchuk2016} with the name of \emph{Segre envelope}.

In this article we are mostly interested  in real analytic Levi-flat subsets which are invariant by  holomorphic foliations in the ambient space. A real analytic Levi-flat subset $H \subset M$ is
 \emph{invariant} by an $n-$dimensional  singular holomorphic foliation  $\mathcal{F}$   on $M$ if
 the Levi leaves are  leaves of $\mathcal{F}$. We also say   that $\mathcal{F}$ is \emph{tangent} to $H$.
If $H$ is invariant by a foliation, the same holds for its $\imath$-complexification:

\begin{prop}\label{fol-inv}
Let $H \subset M$ be a  real analytic Levi-flat subset of $\LL$-dimension $n$, where
 $M$ is a complex manifold of dimension $N$. If $H$ is
  invariant by an $n$-dimensional  holomorphic foliation $\mathcal{F}$ on $M$, then its $\imath$-complexification $H^\imath$ is also invariant by $\mathcal{F}$.
\end{prop}
\begin{proof}
We have $\mathcal{F}|_{H_{reg}}=\mathcal{L},$ where $\mathcal{L}$ is the Levi foliation. The problem is local, so we can work in  a local trivialization    \eqref{levi-local-form}, in which
the $\imath$-complexification is defined by $z'' = 0$
and the   Levi leaves are given by
 $\{ z_{n+1}=c, z''=0\}$, where $c \in \mathbb{R}$.
 Let $\vec{v}=(v_1,...,v_{n+2},...,v_N)$ be a  local vector field  tangent to $\mathcal{F}$. For each $i=n+2,..., N$, every $\zeta=(z_1,...,z_{n})\in \mathbb{C}^{n}$ and $z_{n+1}\in \mathbb{C}$ sufficiently small, it holds
$$v_i(\zeta, Re(z_{n+1}),0,...,0)\equiv0,$$
and thus
$$v_i(\zeta,z_{n+1},0,...,0) \equiv 0.$$
This says that   $H^{\imath}$ is invariant by $\vec{v}$.
\end{proof}
When $H$ is invariant by the foliation $\F$, we denote by $\mathcal{F}^{\imath}=\mathcal{F}|_{H^{\imath}}$ the restriction of $\mathcal{F}$ to $H^{\imath}.$ Note that $\mathcal{F}^{\imath}$ has codimension one in $H^{\imath}.$

\begin{prop}\label{Hc-em-V-V}
 Let $H$ be a germ of real analytic Levi-flat subset.  Then $H^{\mathbb{C}}$ is a subset of $H^{\imath} \times H^{\imath*}$ of complex codimension one.
\end{prop}
\begin{proof} Since $H \subset H^{\imath}$, it is a consequence of the comments in Section \ref{preliminaries} that
\[H^{\C} \subset (H^{\imath})^{\C} = H^{\imath} \times H^{\imath*}.\]
Now, this inclusion must be proper since, otherwise, given a
defining map  $\bs{\phi}$   for $H$, the complexification $\bs{\phi}^{\C}$ would vanish over $H^{\imath} \times H^{\imath*}$, which would imply
that $\bs{\phi}$ itself  would vanish
over  $H^{i}$. Finally,   if $L$ is the closure of a Levi leaf of $H$, which is an analytic set of dimension
  $\diml H$ (see Proposition
\ref{closed-levi-leaves} below), then $L \times L^{*} = L^{\C} \subset H^{C}$.
That is, $H^{\C}$ contains infinitely many complex varieties
  of codimension two in $H^{\imath} \times H^{\imath*}$. This implies that the codimension of
$H^{\mathbb{C}}$ in $H^{\imath} \times H^{\imath*}$  is strictly lower than two, which gives the result.
\end{proof}


Denote by $\pi_1:H^{\mathbb{C}}\rightarrow H^{\imath}$ and $\pi_2:H^{\mathbb{C}}\rightarrow H^{\imath *}$ the restrictions of the two canonical  projections to $H^\mathbb{C}\subset\mathbb{C}^{N}\times \mathbb{C}^{N*}  \simeq  \mathbb{C}^{N}\times \mathbb{C}^{N}\rightarrow \mathbb{C}^{N}.$
The following fact appeared in the proof of Theorem \ref{Levi-viz-Hi}. Its usefulness motivates an
 explicit statement:

\begin{prop}
\label{lema-hip alg}
Let $H$ be a germ of real analytic Levi-flat subset. Then given $p\in \overline{H_{reg}},$ we have  $$\pi_1(H^{\mathbb{C}}_{(p,\bar{p})})=H_p^{\imath} \quad \text{and} \quad \pi_2(H^{\mathbb{C}}_{(p,\bar{p})})=H_{\bar{p}}^{\imath *},$$
where the sets involved are germs of $H^{\mathbb{C}}$, $H^{\imath}$ and $H^{\imath*}$ at $(p,\bar{p})$, $p$ and $\bar{p}$, respectively.
\end{prop}


\section{Segre varieties of Levi-flat subsets}
\label{segre-section}

Let $H$ be a germ of real analytic Levi-flat subset at $(\C^{N},0)$,
$\bs{\phi}=(\phi_1,...,\phi_k) \in  (\A_{N \R})^{k}$  be a generating map and $U$ be a $\bs{\phi}$-reflexive neighborhood.

\begin{ddef}\label{def-segre}
\emph{For each $p \in H^{\imath}\cap U$, the set
$$\Sigma_p(U, \bs{\phi}):=\{z\in U;\bs{\phi}(z,\bar{p})=0\}\cap H^{\imath} \subset U \cap H^{\imath}$$
is called \emph{Segre variety} at $p$ associated to the  generating map $\bs{\phi}$ and to the $\bs{\phi}$-reflexive neighborhood $U$.}
\end{ddef}

The Segre variety $\Sigma_p(U, \bs{\phi}) \subset U$ is a closed analytic set  that contains $p$ if and only if $p \in H$. It does not depend on the generating map and on the neighborhood $U$ of $0 \in \mathbb{C}^N$ in the following sense: if $\bs{\psi}$ is another generating map of $H$ and $V$ is a $\bs{\psi}$-reflexive neighborhood of $0 \in \mathbb{C}^N,$ then there exists a neighborhood of the origin  $W\subset V \cap U$ such that whenever $p\in W \cap H^{\imath}$ it holds $\Sigma_p(U,\bs{\phi})\cap W=\Sigma_p(V,\bs{\psi})\cap W$.
In particular, the germ at $p$ of the Segre variety is well defined. It will be  denote by $\Sigma_p$. It contains $p$
if and only if $p \in H.$

Recall that, by Proposition \ref{Hc-em-V-V}, we have $H^{\mathbb{C}}\subset H^{\imath}\times H^{\imath *}$.  Let $\pi_1 :H^{\mathbb{C}}\rightarrow H^{\imath}$ and $\pi_2 :H^{\mathbb{C}}\rightarrow H^{\imath *}$ be the canonical projections. For $p \in  H^{\imath}$, if we identify $H^{\imath} \times \{p\} \simeq  H^{\imath}$, we have, by \eqref{complexification-function} and \eqref{mirror-function},
\begin{equation}
\label{segre-pi2}
 \pi_{2}^{-1}(\bar{p}) = \{z \in H^{\imath}; \bs{\phi}^{\C}(z,\bar{p}) = 0\} = \Sigma_p .
 \end{equation}
Similarly, under the identification
$\{p\} \times H^{\imath*}  \simeq  H^{\imath*}$, we have that
\begin{equation}
\label{segre-pi1}
\pi_{1}^{-1}(p) = \{w \in H^{\imath*}; \bs{\phi}^{\C}(p,w) = 0\} = \Sigma_p^{*} .
 \end{equation}

We have the following result:

\begin{prop}\label{fornaes} Let $W \subset U \subset \mathbb{C}^{N}$ be an open set and
 $\bs{\phi}(z,\bar{z})$ be a real analytic map in $U$. Suppose that    $L\subset W$ is a complex variety such that $\bs{\phi}(z,\bar{z}) = 0$ for all $z\in L$. Then, for each fixed $p\in L$, we have
 $\bs{\phi}(z,\bar{p})= 0$ for all $z\in L$.
\end{prop}
\begin{proof} Without loss of generality, we can suppose that $W = U$ and that $U$ is $\bs{\phi}$-reflexive. Let $H = \{ \bs{\phi}(z,\bar{z}) = 0 \} \subset U$.
Our hypothesis is that $L \subset H$. Taking complexifications, we find $L \times L^{*} \subset H^{\C} \subset
\C^{N} \times \C^{N*}$. Given $p \in L$, we have
$$L \times \{\bar{p}\} \subset H^{\C} \cap
 (\C^{N} \times \{\bar{p}\}) = \{ \bs{\phi}^{\C}(z,\bar{p}) = 0 \}.$$
 This is equivalent to $L   \subset   \{ \bs{\phi}(z,\bar{p}) = 0 \}$, which is the desired result.
\end{proof}

%


As a consequence, if $H$ is a Levi-flat subset and $L_p$ is the Levi leaf at $p\in \ov{H_{reg}}$, then $L_p \subset \Sigma_p$, which gives
${\rm codim\,}_{\mathbb{C},H^{\imath}}(\Sigma_p) \leq  {\rm codim\,}_{\mathbb{C},H^{\imath}}(L_p) = 1$. This remark motivates the following definition:

\begin{ddef}
\emph{Let $H$ be a germ of real analytic Levi-flat subset.  The point $p\in H $ is said to be \emph{Segre degenerate} or simply \emph{$S$-degenerate} if
$${\rm codim\,}_{\mathbb{C},H^{\imath}} (\Sigma_p)=0.$$
When ${\rm codim\,}_{\mathbb{C},H^{\imath}} (\Sigma_p) =1,$ the point $p\in H$ is called \emph{Segre ordinary} or \emph{$S$-ordinary}. We denote by $S_d$ the  set of  $S$-degenerate points  of $H.$}
\end{ddef}

For a germ $\phi \in \A_{N \R}$ and for a $\phi$-reflexive neighborhood $U$,   equation \eqref{complexification-function} gives that, whenever
$(p,\bar{q}) \in U \times U^{*}$,
\[ \phi^{\C}(q,\bar{p}) = 0 \quad  \Leftrightarrow \quad  \xbar{\phi^{\C}(q,\bar{p}) = 0} \quad \Leftrightarrow \quad  \phi^{\C}(p,\bar{q}) = 0  .\]
This applied to the components of a generating map $\bs{\phi}$ of a Levi-flat subset $H$ and
 to a $\bs{\phi}$-reflexive neighborhood $U$ gives the following:
\begin{prop} \label{obs-seg2}
We have $q\in \Sigma_p(U,\bs{\phi})$ if and only if $p\in \Sigma_q(U,\bs{\phi}).$ In particular, if $p\in S_d,$ then $p\in \Sigma_q$ for every $ q\in U \cap H^{\imath}.$
\end{prop}

We have the following proposition:

\begin{prop}\label{anal-comp}
 $S_d$ is a complex analytic variety.
\end{prop}

\begin{proof}
Following the above notation,
we have
\begin{eqnarray}
\label{s-detenerate-definition}
S_d &= &\{p\in U \cap H^{\imath}; \bs{\phi}^{\C}(z,\bar{p})=0 \ \forall \ z \in U \cap H^{\imath}\} \medskip \\
 &= &
 \{p\in U \cap H^{\imath};  \bs{\phi}^{\C}(p,\bar{z}) = 0 \  \forall \ z \in U \cap H^{\imath}\},  \nonumber
\end{eqnarray}
 and then
$$S_d=\left(\displaystyle\bigcap_{z\in U \cap H^{\imath}}\{\bs{\phi}^{\C}(p,\bar{z}) =0\}\right)\cap H^{\imath}.$$
This defines $S_d$ as a  complex analytic set.
\end{proof}

It is worth commenting that   $S_d$ is a proper subset of  $H^{\imath}$. Indeed,    otherwise,
by  \eqref{s-detenerate-definition},
$\bs{\phi}^{\C}$ would vanish over $H^{\imath} \times H^{\imath *}$. This would happen if and only if
 $\bs{\phi}^{\C} \equiv 0$, which is impossible.
It is a known fact that   the set of $S$-degenerated points
of a Levi-flat hypersurface  form a complex subvariety of codimension at least two contained in $H_{sing}$  \cite{lebl2013}. For
Levi-flat subsets we can state the following:

\begin{prop}
\label{codim2}
$S_d$ has codimension at least two in  $H^{\imath}.$
\end{prop}

\begin{proof}
We first suppose that  $n=\diml H =1$, so that    $dim_\mathbb{R}H=3$ and $dim_\mathbb{C}H^{\imath}=2.$
By contradiction, suppose that there exists a one-dimensional irreducible component $\Gamma\subset S_d$. We have  $\Sigma_p=H^{\imath}$ for every $p\in \Gamma$.
As before, let $\pi_2 :H^{\mathbb{C}} \subset H^{\imath}\times H^{\imath *} \rightarrow H^{\imath *}$ be the projection in the second coordinate.
Then,
by \eqref{segre-pi2},
we have $\pi_2^{-1}(\bar{p})  \simeq \Sigma_p  = H^{\imath}$   for every $p\in \Gamma$. Therefore $\pi_2^{-1}(\Gamma^*)= H^{\imath} \times \Gamma^* $ is  a three-dimensional variety. On the other hand,   $H^{\mathbb{C}}$ is irreducible and thus   $\pi_2^{-1}(\Gamma^*)= H^{\mathbb{C}}$, which gives
$\Gamma^*= \pi_2(H^{\mathbb{C}}) = H^{\imath *}$.
 This is a contradiction, since $\Gamma^{*}$ is properly contained in $H^{\imath *}$.

 The general case $n=\diml H >  1$ follows from the particular one by taking planar sections.
 Consider  a complex plane $\alpha$ of codimension $n-1$ simultaneously transversal to $H,$ $H^{\imath}$ and $H_{sing}$. The sets
     $H_\alpha=H\cap \alpha$ and $H_\alpha^{\imath}=H^{\imath}\cap \alpha$   have dimensions $dim_\mathbb{R}H_\alpha=3$ and $dim_\mathbb{C}H_\alpha^{\imath}=2.$ By the minimality property, we have that  $H^{\imath}_\alpha=(H_\alpha)^{\imath}$ is the $\imath$-complexification of $H_\alpha$.
Let $\bs{\phi}$ be a defining map for $H$.   If $(S_{\alpha})_d$ denotes the set of $S$-degenerated points of $H_{\alpha}$, we have
$$(S_\alpha)_d=\{p\in H_{\alpha}^{\imath}; \bs{\phi}|_{\alpha}(z,\bar{p})\equiv 0 \text{ on } H^{\imath}_\alpha \} \supseteq \{p\in H^{\imath};\bs{\phi}(z,\bar{p})\equiv0 \text{ on } H^{\imath}  \}\cap \alpha=S_d\cap \alpha.$$
The particular case gives that   $(S_\alpha)_d$  is formed by isolated points, which
is enough to conclude that ${\rm codim}_{\mathbb{C},H^{\imath}}S_d\geq 2$.
 \end{proof}

Levi leaves of a real analytic Levi-flat hypersurface    are closed analytic varieties.
The same hods for Levi-flat subsets:

\begin{prop}
\label{closed-levi-leaves}
The Levi leaves of a germ of Levi-flat subset are closed analytic sets.
\end{prop}
\begin{proof}
Indeed, by   Proposition \ref{codim2},
every Levi leaf   contains   $S$-ordinary points.
Thus, if $p  \in H_{reg}$ is $S$-ordinary   and $L_{p}$ is the corresponding Levi leaf, we have $ \dim_\mathbb{C}L_p = \dim_\mathbb{C} \Sigma_p$. Since $L_p \subset \Sigma_p$, we conclude that  $L_{p}$ is a component of the analytic set $\Sigma_p$.
\end{proof}

\begin{remark} \emph{ For a germ of  real analytic Levi-flat subset $H$ at $(\C^{N},0)$,  a   point $p \in H_{sing}$ is said to be \emph{dicritical} if it belongs to (the closure of) infinitely many leaves of $\LL$.
The main result in \cite{pinchuk2016} states
that the notions of dicriticalness  and Segre degeneracy coincide for real analytic Levi-flat subsets.
}\end{remark}


\section{Levi flat subsets in projective spaces}
\label{integration-section}

In this section we  present some results on real analytic Levi-flat subsets in the complex projective space
$\mathbb{P}^{N} = \mathbb{P}^{N}_\mathbb{C}$.
If $H \subset \mathbb{P}^{N}$ is a real analytic variety, then the natural projection
$$\sigma: \mathbb{C}^{N+1}\setminus \{0\}\rightarrow \mathbb{P}^{N}$$
  identifies $H$ with   the complex cone
\begin{equation}\label{cone}
H_{\kappa}:=\{z\in \mathbb{C}^{N+1}\setminus \{0\};\sigma(z)\in H\}\cup \{0\},
 \end{equation}
 which is a    real analytic subvariety in $\mathbb{C}^{N+1} \setminus \{0\}$.  When $H$ is Levi-flat,  $H_{\kappa}$ naturally inherits the Levi structure of $H$ and $\diml H_{\kappa} = \diml H + 1$. We have that $H$ is algebraic if and only if $H_{\kappa}$ is analytic at $0 \in \mathbb{C}^{N+1}$   \cite[Proposition 2.1]{jiri2012}. Thus, in the real algebraic case,
 some of the local constructions done so far can be repeated   for
 the germ of $H_{k}$ at $(\mathbb{C}^{N+1},0)$.

For instance, we can extend the construction of  the (extrinsic) complexification for a real projective algebraic variety
$H \subset \mathbb{P}^{N}$.   Consider the ideal $\mathcal{I}(H_{\kappa})$ in $\mathbb{C}[z,\bar{z}]$, where
 $z = (z_{1},\ldots,z_{N+1})$ are coordinates of $\mathbb{C}^{N+1}$, and take a system of generators $ \phi_1,...,\phi_k$, where, for $j=1,...,k$, each $\phi_j $  is a bihomogeneous polynomial of bidegree
  $(d_j,d_j)$ in the variables $(z,\bar{z})$.
 Their complexifications   define  a complex variety $H_{\kappa}^{\mathbb{C}}$ in $\mathbb{C}^{N+1}\times \mathbb{C}^{N+1},$ which goes down   to an algebraic subvariety $H^{\mathbb{C}}\subset \mathbb{P}^{N}\times \mathbb{P}^{N}$ called \emph{(extrinsic) projective complexification} of $H$.
Note that $H^{\mathbb{C}}$ inherits   the properties of the local complexification $H_{\kappa}^{\mathbb{C}}$.  We summarize this in the following:

\begin{prop}
\label{algebraic-complexification}
Let $H\subset \mathbb{P}^{N}$ be a real algebraic variety. Then $H^\mathbb{C} \subset \mathbb{P}^N \times \mathbb{P}^N$ is a complex algebraic variety, which is irreducible if and only if
   $H$ is.
\end{prop}

We now    examine
the intrinsic complexification $H^{\imath}$ of a real analytic Levi-flat subset  $H \subset \mathbb{P}^N$.
In principle, by pasting local $\imath$-complexifications, we build $H^{\imath}$  as a complex analytic variety of dimension $\diml H+1$  defined in an
open neighborhood of $\ov{H}_{reg}$. When $H$ is algebraic,   $H^{\imath}$  extends to an  algebraic subset of $\mathbb{P}^{N}$, as shown in:

\begin{prop}\label{hi-alg-2}
Let $H\subset \mathbb{P}^{N}$ be an irreducible real algebraic  Levi-flat subset of $\LL$-dimension $n$. Then its $\imath$-complexification $H^{\imath}$ extends to an $(n+1)-$dimensional algebraic variety in $\mathbb{P}^{N}.$
\end{prop}
\begin{proof}
We associate to $H$ its  projective  cone $H_{\kappa}$, which is analytic and irreducible as a germ at $(\mathbb{C}^{N+1},0)$. Let $H_{\kappa}^{\mathbb{C}}$   denote its   complexification at
$(\mathbb{C}^{N+1}\times\mathbb{C}^{N+1},0)$. By Proposition \ref{lema-hip alg}, we have $\pi_1(H_{\kappa}^{\mathbb{C}})=H_{\kappa}^{\imath},$ where $H^{\imath}_{\kappa}$ is the   $\imath$-complexification of $H_{\kappa}$.
By Proposition \ref{algebraic-complexification},
  $H^{\mathbb{C}}\subset \mathbb{P}^{N}\times \mathbb{P}^{N}$   is   complex algebraic   and so is its
image $\pi^{\mathbb{P}}_1(H^{\mathbb{C}})\subset \mathbb{P}^{N}$  by  the projection
$\pi^{\mathbb{P}}_1: \mathbb{P}^{N}\times \mathbb{P}^{N} \to  \mathbb{P}^{N}$
in the first coordinate.
 Note that the cone associated with $\pi^{\mathbb{P}}_1(H^{\mathbb{C}})$ is   $(\pi^{\mathbb{P}}_1(H^{\mathbb{C}}))_{\kappa}=\pi_1(H_\kappa^{\mathbb{C}}) = H_{\kappa}^{\imath}.$
Finally, $H_{\kappa}^{\imath}$  is  the cone of an irreducible algebraic variety in $\mathbb{P}^{N}$  of dimension   $n+1$ which contains $H$. The result follows from the uniqueness of the
intrinsic complexification as a germ around $H$.
 \end{proof}

Next we look at Segre varieties of a Levi-flat  algebraic subset $H$. We identify $H$ with its algebraic cone $H_{\kappa}$ at $(\mathbb{C}^{N+1},0)$ and take a system of bihomogeneous generators $ \phi_1,..., \phi_k \in  \mathbb{C}[z,\bar{z}]$ for the ideal $\mathcal{I}(H_{\kappa})$.  By Proposition \ref{hi-alg-2}, the $\imath$-complexification $H^{\imath}_{\kappa}$ is algebraic. It then follows from Definition \ref{def-segre} that the Segre varieties of $H_{\kappa}$ are algebraic. An arbitrary Levi leaf of $H_{\kappa}$ contains S-ordinary points and, at each of these points,
it is a component of the corresponding Segre variety. This gives the following:

\begin{prop}
\label{cor-folhas-alg}
The Levi leaves of a    real algebraic   Levi-flat subset in $\mathbb{P}^{N}$ are algebraic.
\end{prop}

As we observed, when a Levi-flat subset  $H  \subset \mathbb{P}^{N}$ is  real analytic,
its $\imath$-complexification in principle is defined in a neighborhood of $\ov{H}_{reg}$. However, in certain cases, we can apply extension results of analytic varieties in order to prove that $H^{\imath}$ extends to an algebraic variety in $\mathbb{P}^N$. For instance, we can use  of the following theorem:

\begin{teo}\emph{(Chow, \cite{chow1969})}\label{chow}
Let $Z\subset \mathbb{P}^{N}$ be an algebraic set of dimension $n$ and $V$ be a connected neighborhood  of $Z$ in $\mathbb{P}^{N}.$ Then any analytic subvariety of dimension higher than $N-n$ in $V$ that intersects $Z$ extends algebraically to $\mathbb{P}^{N}.$
\end{teo}

This allows us to state the following extension result for the $\imath$-complexification $H^{\imath}$:
\begin{prop}\label{H-algebrico}
Let $H\subset \mathbb{P}^{N}$ be a real analytic Levi-flat subset of   $\diml H =n$ such that $N>3$ and $n>\dfrac{N-1}{2}.$ If the Levi foliation has an algebraic leaf, then $H^{\imath}$ extends to an algebraic variety in $\mathbb{P}^{N}.$
\end{prop}

\begin{proof}
We have $\dim_\mathbb{C} L=n,$ where $L$ is the Levi leaf which supposed to be algebraic, and $\dim_\mathbb{C} H^{\imath}=n+1.$ Since $n>(N-1)/2,$ we find $\dim_\mathbb{C} H^{\imath} = n+1>N-n$.  The result then follows from Chow's Theorem.
\end{proof}

 A  foliation of codimension one in $\mathbb{P}^{N}$ tangent to an algebraic Levi-flat hypersurface has a  rational first integral \cite[Theorem 6.6]{jiri2012}. We can state a version of this result in  the context of  Levi-flat subsets. We consider a real analytic Levi flat subset $H \subset \mathbb{P}^{N}$ of
$\diml H = n$, invariant by an $n-$dimensional holomorphic foliation $\F$. By Proposition
\ref{fol-inv}, $H^{\imath}$ is invariant by $\F$. We will be mostly concerned with $\mathcal{F}^{\imath}:=\mathcal{F}|_{H^{\imath}},$ which is a codimension one foliation on $H^{\imath}$,
which in principle is a    singular  variety.
We  make use of the following  result on the integrability of foliations in projective manifolds:
\begin{teo} \emph{(X. G\'omex-Mont, \cite{gomez1989})}
\label{gm}
Let $\mathcal{F}$ be a singular holomorphic foliation  of   codimension $q$ on an irreducible projective manifold $M$. Assume that every leaf $L$ of $\mathcal{F}$ is a quasi-projective subvariety of $M$. Then there exist a projective manifold  $X$ of dimension $q$ and a rational map $f:M\rightarrow X$ such that the leaves of $\mathcal{F}$ are contained in the fibers of $f$.
\end{teo}

We also need the following   generalization of Darboux-Jouanolou Theorem \cite{jouanolou1979}:
\begin{teo}\emph{(E. Ghys, \cite{ghys2000})}
\label{GHYS}
Let $\mathcal{F}$ be a singular holomorphic foliation of codimension one  on a smooth, compact and connected analytic complex manifold. If $\mathcal{F}$ has infinitely many closed leaves, then $\mathcal{F}$ has a meromorphic first integral and, therefore, all its leaves are closed.
\end{teo}

In order to apply the above theorems, we have to  desingularize   the $\imath$-complexification   $H^{\imath}$ using
 Hironaka's Dessingularization's Theorem
\cite{hironaka1964}:
there exists  a manifold $\tilde{H^{\imath}}$ and a proper bimeromorphic morphism
$\pi:\tilde{H^{\imath}}\rightarrow H^{\imath}$ such that:
\begin{enumerate}[(i)]
\item  $\pi:\tilde{H^{\imath}} \setminus (\pi^{-1}( \sing (H^{\imath})) \medskip
\rightarrow  H^{\imath} \setminus \sing(H^{\imath})$ is an isomorphism.
\item   $\pi^{-1}(\sing(H^{\imath}))$ is a simple normal crossing divisor.
 \end{enumerate}
Note that if the real analytic Levi-flat subset $H \subset \mathbb{P}^{N}$ is tangent to an abient foliation $\mathcal{F}$   on $\mathbb{P}^{N}$, then
  $\mathcal{F}^{\imath}$, being   the restriction of   $\mathcal{F}$ to $H^{\imath}$,   lifts by
 the desingularization map  to a foliation $\tilde{\mathcal{F}}^{\imath}$ on  $\tilde{H^{\imath}}$.

We then have the main result of this section:

\begin{prop}\label{teoremaA}
 Let $H\subset \mathbb{P}^{N}$ be a real analytic Levi-flat subset of $\diml H = n$ invariant by a holomorphic foliation $\mathcal{F}$ in $\mathbb{P}^{N}$. Suppose that the $\imath$-complexification $H^{\imath}$ extends to an algebraic variety in  $\mathbb{P}^{N}$ --- which happens, for instance, if  $N>3$ and $n> (N-1)/2$. If the Levi foliation $\LL$ has infinitely many algebraic leaves, then  $\mathcal{F}^{\imath} = \mathcal{F}|_{H^{\imath}}$ has  a rational first integral.
\end{prop}

\begin{proof}
Let $\pi:\tilde{H^{\imath}}\rightarrow H^{\imath}$  be a desingularization map.
$H^{\imath}$ is compact and so is $\tilde{H}^{\imath}$. We lift $\mathcal{F}^{\imath}$  to an $n-$dimensional foliation $\tilde{\mathcal{F}}^{\imath}$  on $\tilde{H^{\imath}}$.
Our hypothesis gives that $\mathcal{F}^{\imath}$ has infinitely many closed leaves and thus
 the same holds for $\tilde{\mathcal{F}}^{\imath}$.
 By Theorem  \ref{GHYS},  $\tilde{\mathcal{F}}^{\imath}$ admits a meromorphic first integral in $\tilde{H^{\imath}}.$ So, all leaves of $\tilde{\mathcal{F}}^{\imath}$ are compact. Besides, their $\pi$-images  are compact leaves of $\mathcal{F}^{\imath}$ in $H^{\imath}$. Finally, by Theorem \ref{gm}, there exists a one-dimensional projective manifold $X$  and a rational map $f:H^{\imath}\rightarrow X$ whose fibers contain the leaves of $\mathcal{F}^{\imath}$. The rational first integral is obtained by composing $f$ with a  non-constant rational map $r:X \rightarrow \mathbb{P}^1$.
\end{proof}

When a Levi-flat subset $H\subset \mathbb{P}^{N}$ is algebraic,
assembling the conclusions of Propositions
\ref{hi-alg-2} and \ref{cor-folhas-alg}, the same argument of the proof of Proposition \ref{teoremaA}  gives
 the following integrability result:

\begin{cor}
 \label{p-indet-2-cor}Let $H\subset \mathbb{P}^{N}$ be an algebraic Levi-flat subset invariant by a holomorphic foliation $\mathcal{F}$. Then $H^{\imath}$ is algebraic and $\mathcal{F}^{\imath}$ has a rational first integral.
\end{cor}


\section{Rational functions and Levi-flat subsets}
 \label{integration1-section}


Let $R$ be a rational function on $\mathbb{P}^{N}$ and $S\subset \mathbb{C}$ be a real algebraic curve. Then $\overline{R^{-1}(S)}$ is a Levi-flat hypersurface \cite[Prop. 5.1]{jiri2012}. An equivalent result --- with a similar proof --- can be stated    in
the context of this paper:

\begin{prop}
Let $X\subset \mathbb{P}^{N}$ be an irreducible $(n+1)$-dimensional algebraic variety, $R$ be a rational function on $X$ and $S\subset \mathbb{C}$ be a real algebraic curve. Then the set $\overline{R^{-1}(S)}$ is an algebraic Levi-flat subset of
$\cl{L}$-dimension $n$ whose $\imath$-complexification is $X$.
\end{prop}

Our goal in this section is  to prove that, with the additional hypothesis that
the Levi-flat subset is tangent to a foliation in the ambient space, a reciprocal of this result can
be proved by adapting the techniques of \cite[Theorem 6.1]{jiri2012}.

Fix the usual notation $\ind(F)$ for the indeterminacy set of a meromorphic function $F$.
We have the following local portrait of Levi-flat subsets tangent to the levels of meromorphic functions:

 \begin{prop}\label{p-indeterminação}
 Let $H$ be a germ of irreducible real analytic Levi-flat subset  of $\diml H = n$ at $(\mathbb{C}^{N},0)$. Suppose that $F$ is a non-constant meromorphic function in $H^{\imath}$, such that ${\rm codim}_\mathbb{C}\, \ind(F)\geq 2$, which is constant along the Levi leaves. If $0\in H \cap \ind(F)$, then there exists an algebraic one-dimensional subset $S\subset \mathbb{C}$ such that $H\subset \overline{F^{-1}(S)}.$
\end{prop}

\begin{proof}
Since the proof goes as that of   \cite[Lemma 5.2]{jiri2012}, we just review its main steps and verify that they adapt to our context.
It is sufficient  to consider the   case  $n=1$, for which
  $\dim_\mathbb{C}H^{\imath}=2$ and $0\in H^{\imath}$ is an isolated point of
indeterminacy of $F$ --- the general case  $n>1$ reduces to this   particular one by
cutting $H$ by an $(n-1)$-plane $\alpha$ in general position, as we did in Proposition \ref{codim2}.
 Write $F=f/g$,   where $f$ and $g$ are holomorphic functions in $H^{\imath}$,
without common factors, and consider the map
\[ \Phi: z \in H^{\imath} \mapsto  \left(f(z),g(z)\right) \in \mathbb{C}^{2} \]
The crucial fact is that  $\Phi(H)$ is semianalytic,   an open subset of
an analytic variety $K$ of the same dimension. In fact,   the map
\[\Psi^\mathbb{C}: (z,w) \in H^{\imath}\times (H^{\imath})^* \mapsto  (f(z),g(z), f^{*}(w),g^{*}(w)) \in \mathbb{C}^{4}\]
is finite and thus, by the  Finite Map Theorem,    $\Psi^\mathbb{C}(H^\mathbb{C})$ is an analytic variety.
Therefore, considering
\[\tilde{\Phi}: z \in H^{\imath} \mapsto (f(z),g(z),\overline{f(z)},\overline{g(z)}) \in \mathbb{C}^{4},\]
we have that $\tilde{\Phi}(H) \subset \Psi^\mathbb{C}(H^\mathbb{C})\cap \Delta$ is   open and thus   it is semianalytic. Note that  $\Psi^\mathbb{C}(H^\mathbb{C})\cap \Delta \subset \C^{4}$ can be defined by functions that depend only on the two first coordinates. Thus,
 taking the projection
$\pi: \C^{4} \to \C^{2}$, $\pi (z_1,z_2,z_3,z_4)=(z_1,z_2)$,
we have that $\Phi(H)=\pi(\tilde{\Phi}(H))\subset \pi_1(\Psi^\mathbb{C}(H^\mathbb{C})\cap \Delta) = K$ is also semianalytic.

Note that $\Phi(H)$ contains infinitely many complex lines through the   origin and thus, if
  $r(z,\bar{z})=\sum_{j,k} r_{jk}(z,\bar{z})$ is a defining function for $K$, written in
 bihomogeneous terms of bidegree $(j,k)$, then
$r_{j,k}(z,\bar{z})\equiv 0$ for all $(j,k)$, meaning  that
 $K$ is    real algebraic.
Next, project the
  algebraic set
$$\{(z,\xi)\in \mathbb{C}^2\times \mathbb{C}; z\in K \ \text{and} \ \xi z_2=z_1\}$$
in the  $\xi$-variable. By   Tarski-Seidenberg Theorem  \cite{dries1998}, this projection is   semialgebraic, so it lies in a   one-dimensional  algebraic set $S\subset \mathbb{C}.$ Thus $K\subset \overline{G^{-1}(S)}$. Since   $\phi(H)\subset K$ and  $F=G\circ \phi$, we conclude  that $H\subset \overline{F^{-1}(S)}.$
 \end{proof}

Remark that if
$X \subset \mathbb{P}^{N}$ is an algebraic complex variety of $\dim_\mathbb{C}X \geq 2,$ then any rational function in $X$ admits points of indeterminacy. This gives us   the following:

\begin{prop}\label{p-indet-2}
 Let $\mathcal{F}$ be a holomorphic   foliation in $\mathbb{P}^{N}$   tangent to a real analytic Levi-flat subset $H$ of $\diml H = n$. Suppose that $\mathcal{F}^{\imath}$ has a
 rational first integral $R.$ Then there exists a real algebraic curve $S\subset \mathbb{C}$ such that $H\subset \overline{R^{-1}(S)}.$
 \end{prop}

 \begin{proof}
Write $\overline{H_{reg}}=\cup_{\ell}L_{\ell}$, where $L_{\ell}$ are irreducible complex analytic subvarieties given by the closures of Levi leaves of $\mathcal{F},$ which are levels of the rational function $R.$ Taking  $p\in \ind(R)$, then $p\in H$, since $p\in \cap\overline{L_{\ell}}$.
Applying Proposition \ref{p-indeterminação} at $p,$ we find a one-dimensional algebraic subset $S\subset \mathbb{C}$ such that, locally, $H\subset\overline{R^{-1}(S)}.$ Since $\overline{H_{reg}}=\cup_{\ell}L_{\ell}$ and $p\in L_{\ell}$ for every $\ell$, then $H\subset\overline{R^{-1}(S)}.$

\end{proof}

With this proposition, we accomplish the proof of   Theorem \ref{int-racional-R}:
\begin{proof} [Proof of Theorem \ref{int-racional-R}]
By Proposition \ref{teoremaA},  $\mathcal{F}^{\imath}=\mathcal{F}|_{H^{\imath}}$ has a rational first integral
in   $H^{\imath}$, say $R$.   The result then follows from   Proposition \ref{p-indet-2}.
\end{proof}

In a similar way, the combination of Proposition \ref{p-indet-2} and the Corollary \ref{p-indet-2-cor}   gives:

\medskip
\noindent
\begin{cor}
\label{coro-im-inv-R}
 Let $H\subset \mathbb{P}^{N}$ an algebraic Levi-flat subset  invariant by a foliation $\mathcal{F}$ in $\mathbb{P}^{N}$.
Then there exist a rational function $R$ in $H^{\imath}$ and a real algebraic curve $S\subset \mathbb{C}$ such that  $H\subset \overline{R^{-1}(S)}.$
\end{cor}


\section {A comment on Brunella's integration techniques}
\label{section-brunella}


 In this section we explain how the techniques of \cite{brunella2011} can be adapted in order to prove Theorem \ref{teo-bru-adapt}.
Recall the conditions of its statement: we have    a germ of real analytic Levi-flat  subset $H$
 at $(\mathbb{C}^{N},0)$,  of $\diml H = n$ and $\text{codim}_{\C}\, \sing(H^{\imath}) \geq 2$, invariant by
a germ of holomorphic foliation   $\mathcal{F}$ of  dimension $n$.
We start by remarking that, by applying  the Transversality Lemma (stated a proved in the Appendix) and taking transverse plane sections,   we can suppose that  $\diml H = 1$ and that  $H^{\imath}$ has an isolated singularity at $0 \in \mathbb{C}^{N}$.
We have   the following
Lemma:

\begin{lem}
Let  $H$ be a real analytic Levi-flat subset  of $\cl{L}$-dimension $1$ at $(\mathbb{C}^{N},0)$ invariant by a germ of holomorphic foliation
 $\mathcal{F}$.
Then,  for each $p\in H^\imath\setminus \{0\},$ the mirror of Segre variety $\Sigma_p^{*}\subset (H^{\imath})^*$ is a non-empty curve invariant by the mirror foliation $\mathcal{F}^{\imath *}.$ Besides, if $p$ and $q$ are on the same leaf of $\mathcal{F}^\imath$, then $\Sigma_{p}^{*} =\Sigma_{q}^{*}.$
\end{lem}

\begin{prova}
The fact that $\Sigma_p^{*}\subset (H^{\imath})^*$ is non-empty for every
 $p\in H^\imath$ sufficiently near $0 \in \C^{N}$ follows from Proposition $\ref{lema-hip alg}$.
Since $\text{codim}_\mathbb{C} \, S_d\geqslant 2$ and $\dim H^\imath = 2$, we can suppose that   $0 \in \C^{N}$ is the only
Segre degenerate point, implying  that   $\Sigma^{*}_p$ is a   curve in $H^{\imath *}$
for each $p\in H^{\imath}\setminus \{0\}$.
Take the two-dimensional   foliation
$\mathcal{F}^{\imath}\times \mathcal{F}^{ \imath  *}$
in $H^\imath  \times  H^{\imath *} $ whose leaf through
$(p,q^*) \in (H^\imath \setminus \{0\})\times (H^{\imath *} \setminus \{0\}) $
is   $L_{p,q^*}=L_p\times L^*_{q^*}$, where $L_p$ denotes the leaf of $\F$ through $p$. Consider the analytic complex set of tangencies
between $\mathcal{F}^{\imath}\times\mathcal{F}^{\imath *}$ and $H^\mathbb{C} \subset H^\imath  \times  H^{\imath *}$,
denoted by $Tang(\mathcal{F}^{\imath}\times\mathcal{F}^{\imath *},H^\mathbb{C}) \subset H^\mathbb{C}$. Since $H^{\Delta}\subset Tang(\mathcal{F}\times\mathcal{F}^*,H^\mathbb{C})$, the minimality of the complexification implies that
 $$(H^{\Delta})^\mathbb{C}=H^{\mathbb{C}}= Tang(\mathcal{F}\times\mathcal{F}^*,H^\mathbb{C}).$$
Denote, as before $\pi_1: H^{\mathbb{C}} \subset H^{\imath} \times H^{\imath *}  \to H^{\imath}$
the projection in the first coordinate. Then, for each $p\in H^{\imath} \setminus \{0\}$, the fiber   $\pi_1^{-1}(p)= \Sigma_p^*$ is a one-dimensional analytic set
tangent to $\mathcal{F}\times \mathcal{F}^{*}$.  Thus $\Sigma_p^*$ is invariant by
$\mathcal{F}^{\imath *}$ and is composed by a finite union of   leaves of $\mathcal{F}^{*}$.
It follows that, for a fixed
leaf $L$ of $\mathcal{F}^{\imath}$, the inverse image $\pi_1^{-1}(L) \subset H^{\mathbb{C}}$ is
  invariant by $\mathcal{F}^{\imath}\times \mathcal{F}^{ \imath  *}$ and has the form    $\pi_1^{-1}(L)=L\times \bigcup_{\lambda \in \Lambda} L_{\lambda}^{*}$,
where the $L_\lambda^*$'s are leaves of $\mathcal{F}^{\imath *}$
 and   $\Lambda$ is a finite set.
  In particularly, if $p$ and $q \in L,$ we have $\Sigma_p^*= \{p\}\times \bigcup_{\lambda \in \Lambda} L_{\lambda}^{*} $ and $\Sigma_q^*= \{q\}\times \bigcup_{\lambda \in \Lambda} L_{\lambda}^{*}.$
Identifying these   with $\bigcup_{\lambda \in \Lambda} L_{\lambda}^{*}$, we obtain $\Sigma_p^*= \Sigma_q^*.$
\end{prova}

 Theorem \ref{teo-bru-adapt}  is  a straight consequence of the
  proposition below, for which the
above lemma is a key ingredient. It  restates   Propositions 2 and 4 of \cite{brunella2011} and its proof follows the very same steps as those in Brunella's paper. The only difference is that here we should also take into account the desingularization divisor of the $\imath$-complexification $H^{\imath}$. The hypothesis on the codimension of $\sing(H^{\imath})$ is   needed in order to apply Levi's extension theorem for meromorphic functions.

\begin{prop}\label{teo-int-1-W_p}
Let $\mathcal{F}$ be a germ of one-dimensional holomorphic foliation at $(\mathbb{C}^{N},0)$   tangent to a germ of analytic real Levi-flat subset $H$ of $\diml H = 1$. Suppose that the $\imath$-complexification $H^{\imath}$ has an isolated singularity at origin and that one of the two following conditions is satisfied:

\begin{enumerate}
\item For every $p\in H^{\imath}\setminus \{0\},$ the mirror of Segre variety  $\Sigma_p^*$ is a proper analytic curve in $H^{\imath *}$ passing through the origin;
\item For every  $p\in H^{\imath},$ the mirror of Segre variety $\Sigma_p^*$
 is a proper analytic curve in $H^{\imath*}$ passing through the origin when $p=0;$
\end{enumerate}
Then $\mathcal{F}^{\imath}$ has a  first integral that is purely meromorphic in case (1) and holomorphic in case (2).
\end{prop}


\section{Examples}
\label{section-examples}


Let $Z$ be a real analytic Levi-flat hypersurface    in a  complex manifold $X$ of $\dim_{\C}X = n+1$.
Let $\mathbb{P}T^{*}X$ be the cotangent bundle projectivization, which is a $\mathbb{P}^{n}$-bundle over $X$ whose dimension  is $N=2n+1$. Denote by $\rho$ the projection $\mathbb{P}T^{*}X\to X$.
 The regular part $Z_{reg}$ of $Z$ can be lifted to $\mathbb{P}T^{*}X$, since, for any $z\in Z_{reg}$,  $$T^{\mathbb{C}}_zZ_{reg}=T_z Z_{reg}\cap  J(T_z Z_{reg})\subset T_z X$$
is a  complex hyperplane.
Let $H_{reg}$ be the lifting of $Z_{reg}$ in $\mathbb{P}T^{*}X$. Fix $y\in\overline{H_{reg}}$ such that $\rho(y)=z\in\overline{Z_{reg}}$.
It follows from \cite{brunella2007} that there exists a neighborhood $V\subset \mathbb{P}T^{*}X$ of $y$ and a germ  of complex variety $Y_{y}$ at $y$ of dimension $n+1$ containing $\overline{H_{reg}}$ on $\mathbb{P}T^{*}X$.
 We have that $H=\overline{H_{reg}}$ is a germ at $y$ of Levi-flat subset of $\diml H = n$ on
 $M = \mathbb{P}T^{*}X$. The gluing of the local varieties $Y_{y}$ produces its
 $\imath$-complexification $H^{\imath}$. By this procedure,   any real analytic Levi-flat
hypersurface in a complex manifold $X$ induces a real analytic Levi-flat subset in $\mathbb{P}T^{*}X$.

\par When $X=\mathbb{P}^{(n+1)}$,
its projectivized cotangent bundle is isomorphic to the incidence variety
$$\Upsilon =\{(p,\alpha)\in\mathbb{P}^{n+1}\times\check{\mathbb{P}}^{n+1}; p\in \alpha\},$$
where $\check{\mathbb{P}}^{n+1}$ denotes the  parameter space of all hyperplanes in $\mathbb{P}^{n+1}$
(see  \cite[p. 27]{pereira2015}). Therefore, when considering a real analytic Levi-flat hypersurface $Z$ in $\mathbb{P}^{n+1}$, what we get is a real analytic Levi-flat subset $H$ in $\Upsilon$. However $\Upsilon$ is not a complex projective space and  our main results on global   integrability cannot be applied in this situation.

A canonical way to generate Levi-flat subsets is by intersecting Levi-flat hypersurfaces with complex analytic subvarieties. The
  examples of real analytic Levi subsets we present below are based on this principle.
\begin{example}
{\rm
	Let $H=\{(z_{1},z_{2},z_{3},z_{4})\in\mathbb{C}^4;\bar{z}_{3}z_{2}-\bar{z}_{2}z_{3}=0,\,\,\,z_{4}=0\}$. Then $H$ is a real analytic Levi-flat subset in $\mathbb{C}^4$, with degenerate singularities along the $z_{1}$-axis. The leaves of the Levi foliation are
	$L_c=\{z_{3}=z_{2}c,\,\,z_{4}=0\}$ for $c\in\mathbb{R}$. Note that the $\imath$-complexification of $H$ is the hyperplane $H^{\imath}=\{z_{4}=0\}$.
	On the other hand, since $H$ is a complex cone in $\mathbb{C}^4\setminus\{0\}$, we get that $H$ induces a Levi-flat subset in $\mathbb{P}^3$ that satisfies the hypothesis of Theorem \ref{int-racional-R}. The foliation $\mathcal{F}$ given by the polynomial 1-form $\omega=z_{2}dz_{3}-z_{3}dz_{2}$ defines
	a holomorphic foliation on $\mathbb{P}^3$ tangent to $H$. Moreover, $\mathcal{F}$ has a rational first integral $R=z_{3}/z_{2}$, which clearly defines a rational first integral on $H^{\imath}$.
}
\end{example}

\begin{example}
{\rm
In $\mathbb{C}^4$ with coordinates  $(z_{1},z_{2},z_{3},z_{4})$,  take $$H=\{z_{1}^2\bar{z}_{3}^2-z_{1}\bar{z}_{3}|z_{2}|^2+z_{1}z_{3}\bar{z}_{2}^2-2|z_{1}|^2|z_{3}|^2+
\bar{z}_{1}\bar{z}_{3}z_{2}^2-z_{3}\bar{z_{1}}|z_{2}|^2+z_{3}^2\bar{z_{1}}^2=0,\,\, z_{4}=0\}.$$
Then $H$ is a real analytic Levi-flat subset  foliated by the 2-planes
$$L_c=\{z_{1}+cz_{2}+c^2z_{3}=0,\,\, z_{4}=0\}$$
 for $c\in\mathbb{R}$. Again, the $\imath$-complexification is $H^{\imath} =\{z_{4}=0\}$.
Naturally, $H$ defines a real analytic Levi-flat subset in $\mathbb{P}^3$ but, in this case, $H$ is not invariant by an
ambient holomorphic foliation. Note that, by elimination of $c$ in the system of equations
	\[ \left\{ \begin{array}{ll}
z_{1}+cz_{2}+c^2z_{3}=0     & \\
dz_{1}+cdz_{2}+c^2dz_{3}=0 &
	\end{array} \right. \]
	we obtain a holomorphic 2-web tangent to $H$ on $\mathbb{P}^3$.
}
\end{example}

\begin{example}
{\rm
We  present next a real analytic non-algebraic Levi-flat subset of $\mathbb{P}^3$ of $\LL$-dimension 1, having
an algebraic $\imath$-complexification   and containing
 infinitely many algebraic leaves in its  Levi foliation. However, it is not invariant by a global holomorphic foliation on $\mathbb{P}^3$.
 This   shows that, in   Theorem \ref{int-racional-R}, the assumption   of the existence of a global foliation is essential in order to  get  semialgebricity. We  adapt an  example  in \cite{lebl2015}, whose construction is summarized in the
  following lemma:
\begin{lem}[\cite{lebl2015}]
\label{lebl_lemma}
Let $S\subset\mathbb{R}^2$ be a connected compact real analytic curve without singularities. Let $\tilde{H}$ be the complex cone defined by
$$\tilde{H}=\{(z_0,z_1,z_2)\in\mathbb{C}^3; \, z_0=z_1x+z_2y\,\, \text{\rm for}\,\,(x,y)\in S\}\cup\{z\in\mathbb{C}^3; \, z_1\bar{z}_2=\bar{z}_1z_2\}.$$
Then $\tilde{H}$ is a real analytic Levi-flat hypersurface in $\mathbb{C}^3\setminus\{0\}$ whose canonical projection   $\sigma(\tilde{H})$ is a real analytic Levi-flat hypersurface in $\mathbb{P}^2$.
Besides, if $S$ is not contained in any proper real algebraic curve in $\mathbb{R}^2$, then $\sigma(\tilde{H})$ is not algebraic.
\end{lem}

Let us now  take the projection
$\upsilon: \mathbb{C}^4\setminus\{0\}\to\mathbb{C}^3\setminus\{0\}$ defined by  $\upsilon(z_0,z_1,z_2,z_3)=(z_0,z_1,z_2)$
and the real analytic complex cone  defined by
$$H'=\upsilon^{-1}(\tilde{H})\cap\{(z_0,z_1,z_2,z_3)\in\mathbb{C}^4\setminus\{0\}; \, z_0z_3 - z_1z_2=0\}.$$  Hence   $H=\sigma(H')$ is a real analytic subvariety in $\mathbb{P}^3$.  We have that $H\subset\mathbb{P}^3$ is Levi-flat with $\diml H = 1$ and
   its intrinsic complexification is the quadric $Q \subset \mathbb{P}^3$ defined by $z_0z_3 - z_1z_2 =0$.
  Moreover, if we
 pick $S \subset \mathbb{R}^2$ real   analytic    but non-algebraic, we obtain    a real analytic   non-algebraic Levi-flat subset  $H\subset\mathbb{P}^3$.

\par Finally, we assert  that   $H$ is not tangent to a one-dimensional holomorphic foliation  on $\mathbb{P}^3$. In fact, without loss of generality and possibly translating  $S$, we assume that for all $x\in\mathbb{R}$ small enough, there exists at least two distinct points $y_{1},y_{2}\in\mathbb{R}$ such that $(x,y_{1}), (x,y_{2})\in S$.
Given such a $x \neq 0$, there are at least two distinct leaves of the Levi-flat subset $H$ passing through $[x:1:0:0] \in H$, corresponding the hyperplanes of equations
 $z_{0} = z_{1} x + z_{2} y_{1}$ and $z_{0} = z_{1} x + z_{2} y_{2}$. Then, around these points, the Levi foliation   cannot be tangent to an ambient   holomorphic foliation.

 \begin{center}
\begin{overpic}[scale=.5]
{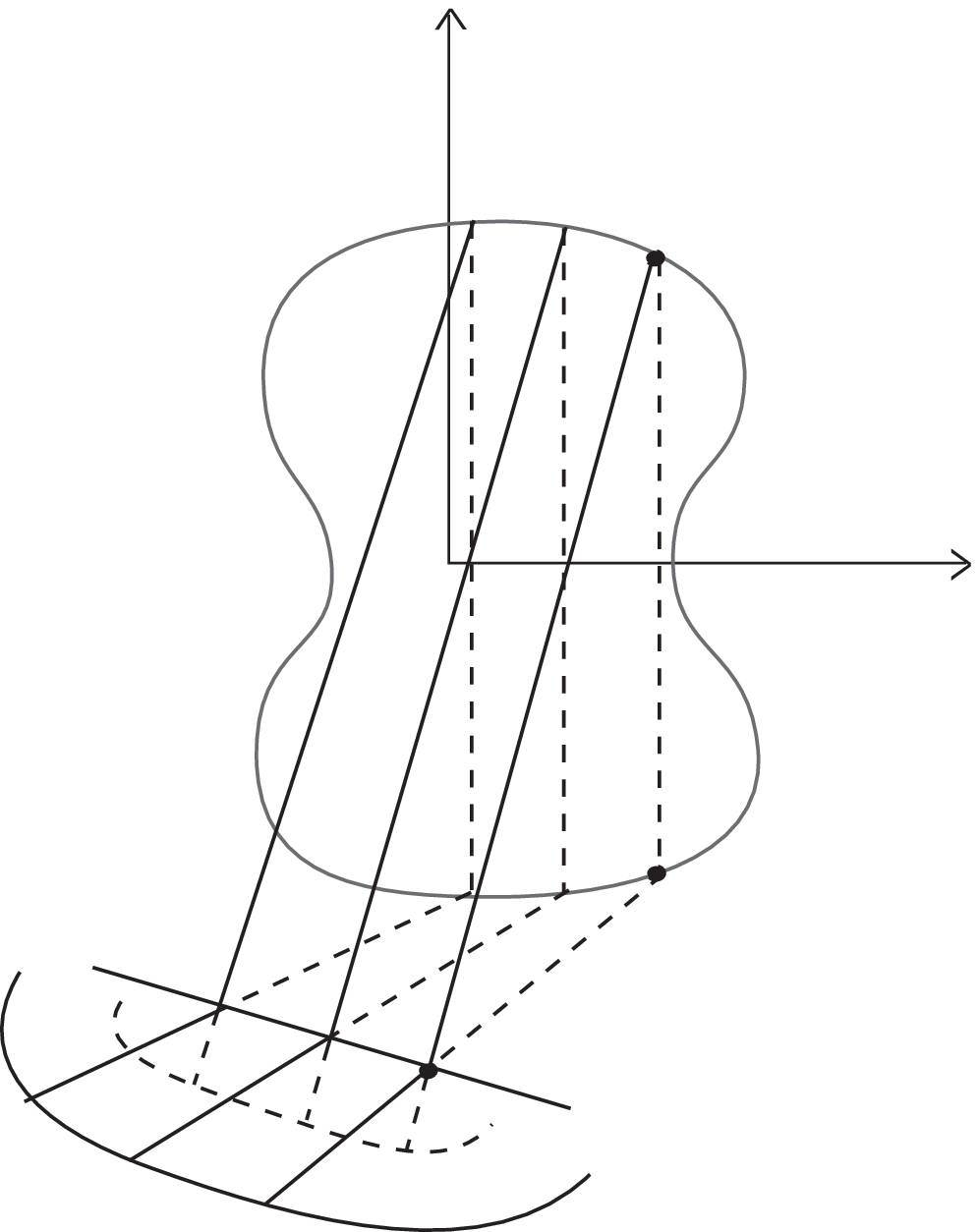}
\put (100,145) {\small $(x,y_1)$}
\put (115,130) {\small $S$}
\put (100,45) {\small $(x,y_2)$}
\put (72,24) {\small $[x:1:0:0]$}
\put (-15,30) {\small $H$}
\end{overpic}
\end{center}

}\end{example}


\appendix
\setcounter{secnumdepth}{0}
\section{Appendix}

Let $M$ be an $N$-dimensional complex manifold whose cotangent sheaf is $\Omega_{M} = \cl{O}(T^{*}M)$.
An $n$-dimensional holomorphic foliation $\F$ on $M$, where $1 \leq n < N$, is the object
defined by   an analytic coherent subsheaf $\cl{C}$ of $\Omega_{M}$ of rank $N-n$
satisfying the following properties (see \cite{suwa1998} for details):
\begin{enumerate}[(i)]
\item  $d \cl{C}_{p} \subset (\Omega_{M} \wedge \cl{C})_{p}$ for every $p \in M \setminus \sing(\cl{C})$ (integrability condition);
\item  $\sing(\Omega_{M} / \cl{C})$ is a set of codimension two. This  is   the \emph{singular set} of $\cl{F}$ and denoted by
$\sing(\cl{F})$.
\end{enumerate}
We call $\cl{C}$   the {\em conormal sheaf} of $\cl{F}$.
Recall that the singular set of a coherent sheaf is the set of points where its stalks fail  to be
free modules over the structural sheaf.
Outside $\sing(\cl{F})$, the conormal sheaf   is  the sheaf of sections of a rank $N-n$
vector subbundle  of $T^{*}M$, defining an integrable holomorphic distribution of subspaces of dimension $N-n$ on  $T^{*}M$ and, thus, a regular   holomorphic foliation of dimension $n$ on $M$.
Then, since ${\rm codim}_{\C}\, \sing(\cl{F}) \geq 2$, the foliation $\F$ is locally induced by
holomorphic   $(N-n)$-forms which are locally decomposable outside $\sing(\cl{F})$ and satisfy the
integrability condition.
We emphasize that  our definition  does not ask   $\cl{F}$ to be a \emph{reduced} foliation. By definition, this happens
when $\cl{C}$ is a \emph{full} sheaf,
that is, whenever $U \subset M$ is  open  and $\omega$ is a holomorphic section of $\Omega_{M}$ over $U$
that is also a section of  $\cl{C}$ over $U  \setminus \sing(\cl{F})$, then it is a section of
$\cl{C}$ over $U$.

We finish by proving a transversality lemma that has been used in  Theorem \ref{teo-bru-adapt}.
First a definition. Let $\F$ be   a germ of singular holomorphic foliation of dimension $n$ at the origin of $M = \C^{N}$
with conormal sheaf   $\cl{C}$,
 where $1< n < N$. Let $\alpha$ be a germ of hyperplane 
 through  $0 \in \C^{N}$
and denote by $\Omega_{\alpha}$  its cotangent sheaf.   We say that $\alpha$ is in \emph{general position} with or \emph{transverse} to $\F$ if the singular set of   $(\Omega_{M} /\cl{C})|_{\alpha} \cong \Omega_{\alpha} /(\cl{C}|_{\alpha})$
  has codimension at least two.
Thus, $\cl{C}|_{\alpha}$ is the conormal sheaf of a foliation  of dimension $n-1$ in $(\alpha,0) \cong (\C^{n-1},0)$ that will be denoted by $\F|_{\alpha}$.
\begin{lem*}[Transversality] Let $\F$ be a germ of singular holomorphic foliation of dimension $n$ at $(\C^{N},0)$.
Then the set of hyperplanes through $0 \in \C^{N}$ transverse to $\F$ form a generic subset
in the Grassmannian ${\rm Gr}_{0}(N-1,N) \cong {\mathbb P}^{N-1}$.
\end{lem*}
\begin{proof}
We have
the following fact: if
$\omega$  is a germ of holomorphic $1-$form at $0 \in \C^{N}$ (not necessarily integrable) with
singular set of codimension at least two, then  the set of hyperplanes through $0 \in \C^{N}$ transverse to
$\omega$ is generic in
${\rm Gr}_{0}(N-1,N) \cong {\mathbb P}^{N-1}$. This is actually a consequence of   the proof of
\cite[Lemma 10]{camacho1992}.
The conormal sheaf   $\cl{C}$  of $\F$ is coherent and thus, generated by finitely many sections
 at $0 \in \C^{N}$, say $k$
 holomorphic $1-$forms $\omega_{1}, \ldots, \omega_{k}$.
For each $i=1,\ldots,k$, we can cancel one-codimensional singular components of  $\omega_{i}$, obtaining  holomorphic   $1-$forms
$\tilde{\omega}_{i}$ such   that $\sing(\tilde{\omega}_{i}) \geq 2$. Note that, since we are not
 assuming that $\F$ is reduced,  each
  $\tilde{\omega}_{i}$  does not necessarily define a section of $\cl{C}$,  yielding however a   section  outside $\sing(\F)$.
The set of hyperplanes transverse to each $\tilde{\omega}_{i}$ is a generic set
$\Gamma_{i} \subset {\rm Gr}_{0}(N-1,N)$. Let $\Gamma_{0}$ denote the generic set of hyperplanes  transverse to $\sing(\F)$ and  consider the set $\Gamma = \cap_{i=0}^{k} \Gamma_{i}$.
Then  $\Gamma \subset {\rm Gr}_{0}(N-1,N)$ is a generic set formed by hyperplanes transverse to $\F$.
In fact, fix $\alpha \in \Gamma$. Let $S_{0} = \sing(\F) \cap \alpha$ and $S_{i} = \sing(\tilde{\omega}_{i}|_{\alpha})$ for $i=1,\ldots,k$. Then $S = \cup_{i=0}^{k}S_{i}$ is a germ of analytic subset in $(\alpha,0) \cong (C^{N-1},0)$ of codimension at least two.
We assert that $\sing(\F|_{\alpha}) \subset S$. Indeed, if $p  \in \alpha \setminus  S$, then $p \not\in \sing(\F)$
and thus there are    $1-$forms $\omega_{i_{1}}, \ldots, \omega_{i_{N-n}}$, all of them non singular at $p$,  such that
 \[ T_{p}\F =\{ \omega_{i_{1}}(p) = \cdots = \omega_{i_{N-n}}(p) = 0 \} .\]
But $H$ is transverse to each  $\tilde{\omega}_{i_{\ell}}$  --- and also to
$\omega_{i_{\ell}}$ --- at $p$, giving that $p$ is not a singular point
for  $\F|_{\alpha}$.
\end{proof}

\bibliographystyle{plain}
\bibliography{referencias}

\medskip \medskip  \medskip
\noindent
Jane Bretas \\
Departamento de F\'isica e Matem\'atica \\
Centro Federal de Educa\c c\~ao Tecnol\'ogica de Minas Gerais \\
Av. Amazonas, 7675
  -- Belo Horizonte, BRAZIL  \\
janebretas@des.cefetmg.br

\medskip  \medskip  \medskip

\noindent
Arturo Fern\'andez-P\'erez\\
Departamento de Matem\'atica \\
Universidade Federal de Minas Gerais \\
Av. Ant\^onio Carlos, 6627  \  C.P. 702  \\
30123-970  --
Belo Horizonte -- MG,
BRAZIL \\
arturofp@mat.ufmg.br

\medskip \medskip \medskip

\noindent
Rog\'erio  Mol  \\
Departamento de Matem\'atica \\
Universidade Federal de Minas Gerais \\
Av. Ant\^onio Carlos, 6627  \  C.P. 702  \\
30123-970  --
Belo Horizonte -- MG,
BRAZIL \\
rmol@ufmg.br

\end{document}